\documentclass[reqno, a4paper]{amsart}
\usepackage[T1]{fontenc}
\usepackage[utf8]{inputenc}
\usepackage{amsmath,amsfonts,amssymb,graphics,amsthm, comment}
\usepackage{textcomp}
\usepackage{upgreek,mathtools,bm}
\usepackage{newtxtext}
\usepackage[varvw]{newtxmath}
\usepackage[colorlinks = true, linkcolor = blue, urlcolor  = blue, citecolor = blue]{hyperref}
\usepackage{enumitem}
\usepackage{xcolor}
\usepackage{tikz-cd}

\theoremstyle{plain}
\newtheorem{theorem}{Theorem}

\newtheorem{lemma}[theorem]{Lemma}
\newtheorem{proposition}[theorem]{Proposition}

\theoremstyle{definition}
\newtheorem{definition}[theorem]{Definition}

\theoremstyle{remark}

\numberwithin{equation}{section}
\numberwithin{theorem}{section}

\newcommand{\R}{\mathbb{R}}
\newcommand{\Rd}{\mathbb{R}^{d}}

\newcommand{\supp}{\operatorname{supp}}

\newcommand{\W}{\mathcal{W}}

\newcommand{\E}{\mathbb{E}}
\newcommand{\MT}{\mathsf{MT}}
\newcommand{\Cpl}{\mathsf{Cpl}}

\newcommand{\bary}{\operatorname{bary}}
\newcommand{\lc}{\preceq_{\textnormal{c}}}
\newcommand{\MCov}{\textnormal{MCov}}

\newcommand{\Law}{\operatorname{Law}}

\renewcommand{\P}{\mathcal{P}}
\newcommand{\PP}{\mathcal{P}}

\begin{document}

\title{The Bass functional of martingale transport}

\author[J.\ Backhoff-Veraguas]{Julio Backhoff-Veraguas}
\address{Faculty of Mathematics, University of Vienna}
\email{julio.backhoff@univie.ac.at}

\thanks{
We thank Ben Robinson for his valuable comments during the preparation of this paper. 
WS and BT %also 
acknowledge support by the Austrian Science Fund (FWF) through projects P 35197 and P 35519, and JB acknowledges support by the FWF through projects Y 00782 and P 36835.
}
 
%\author[M.\ Beiglb\"ock]{Mathias Beiglb\"ock}
%\address{Faculty of Mathematics, University of Vienna}
%\email{mathias.beiglboeck@univie.ac.at}

\author[W.\ Schachermayer]{Walter Schachermayer}
\address{Faculty of Mathematics, University of Vienna}
\email{walter.schachermayer@univie.ac.at}

\author[B.\ Tschiderer]{Bertram Tschiderer}
\address{Faculty of Mathematics, University of Vienna}
\email{bertram.tschiderer@univie.ac.at}

\begin{abstract} 

An interesting question in the field of martingale optimal transport, is to determine the martingale with prescribed initial and terminal marginals which is most correlated to Brownian motion. Under a necessary and sufficient irreducibility condition, the answer to this question is given by a \emph{Bass martingale}. At an intuitive level, the latter can be imagined as an order-preserving and martingale-preserving space transformation of an underlying Brownian motion starting with an initial law $\alpha$ which is tuned to ensure the marginal constraints.

In this article we study how to determine the aforementioned initial condition $\alpha$. This is done by a careful study of what we dub the \emph{Bass functional}. In our main result we show the equivalence between the existence of minimizers of the Bass functional and the existence of a Bass martingale with prescribed marginals. This complements the convex duality approach in a companion paper by the present authors together with M.\ Beiglb\"ock, with a purely variational perspective. We also establish an infinitesimal version of this result, and furthermore prove the displacement convexity of the Bass functional along certain generalized geodesics in the $2$-Wasserstein space.
 
\bigskip

\noindent\emph{Keywords:} Optimal transport, Brenier's theorem, Benamou--Brenier, Stretched Brownian motion, Bass martingale. 

\smallskip

\noindent\emph{Mathematics Subject Classification (2010):} Primary 60G42, 60G44; Secondary 91G20.
\end{abstract}

\maketitle

\section{Introduction}

\subsection{Martingale optimization problem}

Let $\mu, \nu$ be elements of $\PP_{2}(\Rd)$, the space of probability measures on $\Rd$ with finite second moments. Assume that $\mu, \nu$ are in convex order, denoted by $\mu \lc \nu$, and meaning that $\int \phi \, d\mu \leqslant \int \phi \, d\nu$ holds for all convex functions $\phi \colon \Rd \rightarrow \R$. As in \cite{BaBeHuKa20,BBST23} we consider the martingale optimization problem 
\begin{equation*} \label{MBMBB} \tag{MBB}
MT(\mu, \nu) \coloneqq 
\inf_{\substack{M_{0} \sim \mu, \, M_{1} \sim \nu, \\ M_{t} = M_{0} + \int_{0}^{t}  \sigma_{s} \, dB_{s}}} 
\mathbb{E}\Big[\int_{0}^{1} \vert \sigma_{t} - I_{d} \vert^{2}_{\textnormal{HS}} \, dt\Big],
\end{equation*} 
where $B$ is Brownian motion on $\Rd$ and $\vert \cdot \vert_{\textnormal{HS}}$ denotes the Hilbert--Schmidt norm. The abbreviation ``MBB'' stands for ``Martingale Benamou--Brenier'' and this designation is motivated from the fact that \eqref{MBMBB} can be seen as a martingale counterpart of the classical formulation in optimal transport by Benamou--Brenier \cite{BeBr99}, see \cite{BaBeHuKa20,BBST23}. The problem \eqref{MBMBB} is equivalent to maximizing the covariance between $M$ and Brownian motion subject to the marginal conditions $M_{0} \sim \mu$ and $M_{1} \sim \nu$, to wit
\begin{equation} \label{MBMBB2}
P(\mu, \nu) \coloneqq 
\sup_{\substack{M_0 \sim \mu, \, M_1 \sim \nu, \\ M_t = M_0 + \int_0^t \sigma_s \, dB_s}}
\mathbb{E}\Big[\int_0^1 \textnormal{tr} (\sigma_t) \, dt\Big].
\end{equation} 
Both problems have the same optimizer and the values are related via 
\[
MT(\mu,\nu) = d + \int \vert y \vert^{2} \, d\nu(y)  - \int \vert x \vert^{2} \, d\mu(x) -2 P(\mu,\nu).
\]
As shown in \cite{BaBeHuKa20}, the problem \eqref{MBMBB} admits a strong Markov martingale $\hat{M}$ as a unique optimizer, which is called the \textit{stretched Brownian motion} from $\mu$ to $\nu$ in \cite{BaBeHuKa20}.

\subsection{Bass martingales and structure of stretched Brownian motion}

Owing to the work \cite{BBST23} it is known that the optimality property of stretched Brownian motion is related to a structural / geometric description. For its formulation we start with the following definition.

\begin{definition} \label{defi:irreducible_intro} 
For probability measures $\mu, \nu$ we say that the pair $(\mu,\nu)$ is \textit{irreducible} if for all measurable sets $A, B \subseteq \Rd$ with $\mu(A), \nu(B)>0$ there is a martingale $X= (X_{t})_{0 \leqslant t \leqslant 1}$ with $X_{0} \sim \mu$, $X_{1} \sim \nu$ such that $\mathbb{P}(X_{0}\in A, X_{1}\in B) >0$. 
\end{definition} 

We remark that in the classical theory of optimal transport one can always find couplings $(X_{0},X_{1})$ of $(\mu,\nu)$ such that $\mathbb{P}(X_{0}\in A, X_{1}\in B) > 0$, for all measurable sets $A, B \subseteq \Rd$ with $\mu(A),\nu(B) > 0$; e.g., by letting $(X_{0},X_{1})$ be independent. In martingale optimal transport this property may fail.

\smallskip

Next we recall the following concept from \cite{Ba83, BaBeHuKa20, BBST23}.

\begin{definition} \label{def:BassMarti_intro} Let $B = (B_{t})_{0 \leqslant t \leqslant 1}$ be Brownian motion on $\Rd$ with $B_{0} \sim \hat{\alpha}$, where $\hat{\alpha}$ is an arbitrary element of $\PP(\Rd)$, the space of probability measures on $\Rd$. Let $\hat{v} \colon \Rd \rightarrow \R$ be convex such that $\nabla \hat{v}(B_{1})$ is square-integrable. We call 
\begin{equation} \label{def:BassMarti_intro_eq} 
\hat{M}_{t} \coloneqq 
\E[\nabla \hat{v}(B_{1}) \, \vert \, \sigma(B_{s} \colon s \leqslant t)]
= \E[\nabla \hat{v}(B_{1}) \, \vert \, B_{t}], \qquad 0 \leqslant t \leqslant 1
\end{equation}
a \textit{Bass martingale} with \textit{Bass measure} $\hat{\alpha}$ joining $\mu = \Law(\hat{M}_{0})$ with $\nu = \Law(\hat{M}_{1})$.
\end{definition}

The reason behind this terminology is that Bass \cite{Ba83} used this construction (with $d=1$ and $\hat{\alpha}$ a Dirac measure) in order to derive a solution of the Skorokhod embedding problem.

\smallskip

In \cite[Theorem 1.3]{BBST23} it is shown that under the irreducibility assumption on the pair $(\mu,\nu)$ there is a unique Bass martingale $\hat{M}$ from $\mu$ to $\nu$, i.e., satisfying $\hat{M}_{0} \sim \mu$ and $\hat{M}_{1} \sim \nu$:

\begin{theorem} \label{MainTheorem} Let $\mu, \nu \in \PP_{2}(\Rd)$ with $\mu \lc \nu$ and assume that $(\mu,\nu)$ is irreducible. Then the following are equivalent for a martingale $\hat{M}=(\hat{M}_{t})_{0 \leqslant t \leqslant 1}$ with $\hat{M}_{0} \sim \mu$ and $\hat{M}_{1} \sim \nu$:
\begin{enumerate}[label=(\arabic*)] 
\item \label{MainTheorem_1} $\hat{M}$ is stretched Brownian motion, i.e., the optimizer of \eqref{MBMBB}.
\item \label{MainTheorem_2} $\hat{M}$ is a Bass martingale.  
\end{enumerate}
\end{theorem}

Since, for probability measures $\mu, \nu \in \PP_{2}(\Rd)$ with $\mu \lc \nu$, stretched Brownian motion always exists by \cite[Theorem 1.5]{BaBeHuKa20}, the above theorem states that the existence of a Bass martingale follows from --- and is in fact equivalent to --- the irreducibility assumption on the pair $(\mu,\nu)$. 

\smallskip

Denoting by $\ast$ the convolution operator and by $\gamma$ the standard Gaussian measure on $\Rd$, we remark that the convex function $\hat{v}$ and the Bass measure $\hat{\alpha}$ from Definition \ref{def:BassMarti_intro} satisfy the identities 
\begin{equation} \label{eq_def_id_bm}
(\nabla \hat{v} \ast \gamma)(\hat{\alpha}) = \mu
\qquad \textnormal{ and } \qquad 
\nabla \hat{v}(\hat{\alpha} \ast \gamma) = \nu.
\end{equation}
We formalize the fundamental relations \eqref{eq_def_id_bm} and their correspondence with Bass martingales in Lemma \ref{lem_eq_co_bm_13} below.

\smallskip

Throughout we write $\gamma^{t}$ for the $d$-dimensional centered Gaussian distribution with covariance matrix $tI_{d}$ and set $\hat{v}_{t} \coloneqq \hat{v} \ast \gamma^{1-t} \colon \Rd \rightarrow \R$, for $0 \leqslant t \leqslant 1$. In these terms, \eqref{def:BassMarti_intro_eq} amounts to
\[
\hat{M}_{t} = \nabla \hat{v}_{t}(B_{t}), \qquad 0 \leqslant t \leqslant 1.
\]

\subsection{Main results} In the following we denote by $\MCov$ the \textit{maximal covariance}  between two probability measures $p_{1},p_{2} \in \PP_{2}(\Rd)$, defined as
\begin{equation} \label{eq_def_mcov}
\MCov(p_{1},p_{2}) \coloneqq \sup_{q \in \Cpl(p_{1},p_{2})} \int \langle x_{1},x_{2} \rangle \, q(dx_{1},dx_{2}),
\end{equation}
where $\Cpl(\mu,\nu)$ is the set of all couplings $\pi \in \PP(\Rd \times \Rd)$ between $\mu$ and $\nu$, i.e., probability measures on $\Rd \times \Rd$ with first marginal $\mu$ and second marginal $\nu$. As is well known, maximizing the covariance between $p_{1}$ and $p_{2}$ is equivalent to minimizing their expected squared distance; see also \eqref{def_eq_was} below.

\begin{definition} We introduce the \textit{Bass functional}
\begin{equation} \label{def.bass.func}
\PP_{2}(\R^{d}) \ni \alpha \longmapsto \mathcal{V}(\alpha) 
\coloneqq \MCov(\alpha \ast \gamma, \nu) - \MCov(\alpha,\mu).
\end{equation}
\end{definition}

In our first main result we derive a novel formulation of problem \eqref{MBMBB2}, which characterizes the Bass measure $\hat{\alpha}$ in \eqref{eq_def_id_bm} as the optimizer of the Bass functional \eqref{def.bass.func}.

\begin{theorem} \label{prop_alpha_vari_mc} Let $\mu, \nu \in \PP_{2}(\Rd)$ with $\mu \lc \nu$. Then
\begin{equation} \label{eq:variational_alpha_mc}
P(\mu,\nu) 
= \inf_{\alpha \in \P_{2}(\R^{d})} \mathcal{V}(\alpha).
\end{equation}
The right-hand side of \eqref{eq:variational_alpha_mc} is attained by $\hat{\alpha} \in \P_{2}(\R^{d})$ if and only if there is a Bass martingale from $\mu$ to $\nu$ with Bass measure $\hat{\alpha} \in \P_{2}(\R^{d})$.
\end{theorem}

The proof of Theorem \ref{prop_alpha_vari_mc} is given in Section \ref{sec2avcob}. In Section \ref{sec3aivot} we will show the following infinitesimal version of Theorem \ref{prop_alpha_vari_mc}, which constitutes our second main result:

\begin{theorem} \label{theo_ms2} Let $(M_{t})_{0 \leqslant t \leqslant 1}$ be an $\R^{d}$-valued martingale bounded in $L^{2}$, which is given by the stochastic integral
\[
M_{t} = M_{0} + \int_{0}^{t} \sigma_{s} \, dB_{s}, \qquad 0 \leqslant t \leqslant 1,
\]
where $(\sigma_{t})_{0 \leqslant t \leqslant 1}$ is a progressively measurable process. Denote by $\mu_{t}$ the law of $M_{t}$. For Lebesgue-a.e.\ $0 \leqslant t \leqslant 1$ we have, for each $\alpha \in \P_{2}(\R^{d})$, the inequality
\begin{equation} \label{eq_m_theo_ms2}
\mathbb{E}\big[ \textnormal{tr}(\sigma_{t})\big]
\leqslant \liminf_{h \rightarrow 0} \tfrac{1}{h}
\Big( \MCov(\alpha \ast \gamma^{h},\mu_{t+h}) - \MCov(\alpha,\mu_{t}) \Big).
\end{equation}
\end{theorem}

We note that, for a Bass martingale $(\hat{M}_{t})_{0 \leqslant t \leqslant 1}$ of the form
\[
d\hat{M}_{t} = \hat{\sigma}_{t}(\hat{M}_{t}) \, dB_{t},
\]
with associated Bass measure $\hat{\alpha} \in \P_{2}(\R^{d})$ and diffusion function $\hat{\sigma}_{t} \colon \Rd \rightarrow \R^{d \times d}$, we have, for Lebesgue-a.e.\ $0 \leqslant t \leqslant 1$, the equality
\[
\mathbb{E}\big[ \textnormal{tr}\big(\hat{\sigma}_{t}(\hat{M}_{t})\big)\big]
= \frac{d}{dt} \, \MCov(\hat{\alpha} \ast \gamma^{t},\hat{\mu}_{t}),
\]
where $\hat{\mu}_{t} = \Law(\hat{M}_{t})$. This exhibits the sharpness of \eqref{eq_m_theo_ms2} and shows that Theorem \ref{theo_ms2} is an infinitesimal analogue of Theorem \ref{prop_alpha_vari_mc}. 

\medskip

In our final main result we discuss convexity properties of the Bass functional $\alpha \mapsto \mathcal{V}(\alpha)$ defined in \eqref{def.bass.func}.

%According to the right-hand side of \eqref{eq:variational_alpha_mc} the minimization of this functional is akin to finding a Bass measure and hence a Bass martingale. When we use the linear structure of $\PP_{2}(\R^{d})$ this functional can fail to be convex (eg.\ if $\mu=\delta_0$ then $\mathcal{V}$ is concave in this sense). However if we use the almost-Riemannian structure of $\PP_{2}(\R^{d})$, provided by the theory of optimal transport (see e.g.\ \cite{AGS08}), we can achieve a positive result: 

\begin{theorem} We have the following results:
\begin{enumerate}[label=(\arabic*)] 
\item If $d = 1$, then $\mathcal{V}$ is displacement convex, i.e., convex along the geodesics given by McCann interpolations \cite{McCgas}.
\item If $d \geqslant 1$, then $\mathcal{V}$ is displacement convex along generalized geodesics with base $\mu$.
\end{enumerate}
\end{theorem}

The proof of this result, together with a discussion on the various forms of convexity stated therein (see e.g.\ \cite{AGS08, Vi03, McCgas}), and a treatment of the strict convexity of $\mathcal{V}$, are deferred to Section \ref{sec4dcotbf}. We merely stress here that the Bass functional fails to be convex, and can even be concave, if we consider convex combinations of measures in the usual linear sense.

\subsection{Related literature} Optimal transport as a field in mathematics goes back to Monge \cite{Mo81} and Kantorovich \cite{Ka42}, who established its modern formulation. The seminal results of Benamou, Brenier, and McCann \cite{Br87, Br91, BeBr99, Mc94, Mc95} form the basis of the modern theory, with striking applications in a variety of different areas, see the monographs \cite{Vi03, Vi09, AmGi13, Sa15}. 

\smallskip

We are interested in transport problems where the transport plan satisfies an additional martingale constraint. This additional requirement arises naturally in finance (e.g.\ \cite{BeHePe12}), but is of independent mathematical interest. For example there are notable consequences for the study of martingale inequalities (e.g.\ \cite{BoNu13,HeObSpTo12,ObSp14}) and the Skorokhod embedding problem (e.g.\ \cite{BeCoHu14, KaTaTo15, BeNuSt19}). Early articles on this topic of \textit{martingale optimal transport} include \cite{HoNe12, BeHePe12, TaTo13, GaHeTo13, DoSo12, CaLaMa14}. The study of irreducibility of a pair of marginals $(\mu,\nu)$ was initiated by Beiglb\"ock and Juillet \cite{BeJu16} in dimension one and extended in the works \cite{GhKiLi19,DeTo17,ObSi17} to multiple dimensions.

\smallskip

Continuous-time martingale optimal transport problems have received much attention in the recent years; see e.g.\ \cite{BeHeTo15, CoObTo19, GhKiPa19,GuLoWa19, GhKiLi20, ChKiPrSo20, GuLo21}. In this paper we concern ourselves with the specific structure given by the martingale Benamou--Brenier problem, introduced in \cite{BaBeHuKa20} in probabilistic language and in \cite{HuTr17} in PDE language, and subsequently studied through the point of view of duality theory in \cite{BBST23}. In the context of market impact in finance, the same kind of problem appeared independently in a work by Loeper \cite{Lo18}. 

\smallskip

It was also shown in \cite{BaBeHuKa20} that the optimizer $\hat{M}$ of the problem \eqref{MBMBB} is the process whose evolution follows the movement of Brownian motion as closely as possible with respect to an \textit{adapted Wasserstein distance} (see e.g.\ \cite{BaBaBeEd19a, Fo22a}) subject to the given marginal constraints.

%This property motivated to call the martingale $\hat{M}$ the ``stretched Brownian motion'' from $\mu$ to $\nu$ in \cite{BaBeHuKa20}.

%\begin{remark} The requirement that the optimizer $\hat{\alpha}$ should have a finite second moment deserves some comment. As shown in \cite[Example 6.7]{BBST23}, there are
%In this case the right-hand side of \eqref{eq:variational_alpha_mc} takes the form ``$\infty - \infty$'' and has to be interpreted in a relaxed way. As the case $\hat{\alpha} \notin \P_{2}(\R^{d})$ seems of limited interest and is rather technical, we have preferred to formulate Theorem \ref{prop_alpha_vari_mc} under the additional assumption that $\hat{\alpha}$ has a finite second moment.
%\end{remark}

{\small \tableofcontents}

\section{Preliminaries}

\noindent In this short section we give a more detailed review of some of the main results in \cite{BBST23}, which will be useful for the coming discussions and proofs.

\subsection{Dual viewpoint}

As established in \cite[Theorem 1.4]{BBST23}, the problem \eqref{MBMBB2} admits a dual formulation with a particularly appealing structure:

\begin{theorem} \label{theorem_new_duality} Let $\mu, \nu \in \PP_{2}(\Rd)$ with $\mu \lc \nu$. The value $P(\mu,\nu)$ of the problem \eqref{MBMBB2} is equal to 
\begin{equation} \label{WeakDual}
D(\mu,\nu) \coloneqq \inf_{\substack{\psi \in L^{1}(\nu), \\ \textnormal{$\psi$ convex}}} \Big( \int \psi \, d\nu - \int (\psi^{\ast} \ast \gamma)^{\ast} \, d \mu \Big)
\end{equation}
and is attained by a convex function $\hat{\psi}$ if and only if $(\mu,\nu)$ is irreducible. In this case, the (unique) optimizer to \eqref{MBMBB}, \eqref{MBMBB2} is given by the Bass martingale with associated convex function $\hat{v} = \hat{\psi}^{\ast}$ and Bass measure $\hat{\alpha} = \nabla (\hat{\psi}^{\ast} \ast \gamma)^{\ast}(\mu)$.
\end{theorem} 

Note that the symbol $\ast$ used as a superscript denotes the convex conjugate of a function. We also remark that attainment of $D(\mu,\nu)$ has to be understood in a ``relaxed'' sense, since the optimizer $\hat{\psi}$ is not necessarily $\nu$-integrable; see \cite[Proposition 4.2]{BBST23}.

\subsection{Static martingale optimal transport} \label{subs_smot}

We fix $\mu, \nu \in \PP_{2}(\Rd)$ with $\mu \lc \nu$ and consider a static / discrete-time version of the continuous-time martingale optimization problem \eqref{MBMBB2}, to wit
\begin{equation} \label{eq_primal} 
\tilde{P}(\mu,\nu) 
\coloneqq  \sup_{\pi \in \MT(\mu,\nu)} \int \MCov(\pi_{x},\gamma) \, \mu(dx).
\end{equation}
%the maximal covariance ($\MCov$) between two probability measures $p_{1},p_{2} \in \PP_{2}(\Rd)$ is defined as
%\begin{equation} \label{eq_def_mcov}
%\MCov(p_{1},p_{2}) \coloneqq \sup_{q \in \Cpl(p_{1},p_{2})} \int \langle x_{1},x_{2} \rangle \, %q(dx_{1},dx_{2}),
%\end{equation}
%where $\Cpl(\mu,\nu)$ is the set of all couplings $\pi \in \PP(\Rd \times \Rd)$ between $\mu$ and $\nu$, i.e., probability measures $\pi$ on $\Rd \times \Rd$ with first marginal $\mu$ and second marginal $\nu$.
The collection of martingale transports $\MT(\mu,\nu)$ consists of those couplings $\pi \in \Cpl(\mu,\nu)$ that satisfy $\bary(\pi_{x}) \coloneqq \int y \, \pi_{x}(dy) = x$, for $\mu$-a.e.\ $x \in \Rd$. Here, the family of probability measures $\{\pi_{x}\}_{x\in\Rd} \subseteq \PP_{2}(\Rd)$ is obtained by disintegrating the coupling $\pi$ with respect to its first marginal $\mu$, i.e., $\pi(dx,dy) = \pi_{x}(dy) \, \mu(dx)$.

\smallskip

By \cite[Theorem 2.2]{BaBeHuKa20} the value $\tilde{P}(\mu,\nu)$ of \eqref{eq_primal} is finite and equals $P(\mu,\nu)$, as defined in \eqref{MBMBB2}. Furthermore, there exists a unique optimizer $\hat{\pi} \in \MT(\mu,\nu)$ of \eqref{eq_primal} and if $(\hat{M}_{t})_{0 \leqslant t \leqslant 1}$ is the stretched Brownian motion from $\mu$ to $\nu$, then the law of $(\hat{M}_{0},\hat{M}_{1})$ equals $\hat{\pi}$.

\smallskip

As already alluded to, maximizing the maximal covariance is equivalent to minimizing the squared quadratic Wasserstein distance, modulo adding constants. More precisely, in the present setting we have the relation
\[
\inf_{\pi \in \MT(\mu,\nu)}  \int \W_{2}^{2}(\pi_{x},\gamma) \, \mu(dx) 
= d + \int \vert y \vert^{2} \, d\nu(y) -2 \tilde{P}(\mu,\nu),
\]
where the quadratic Wasserstein distance $\W_{2}(\, \cdot \, , \, \cdot \,)$ between two probability measures $p_{1},p_{2} \in \PP_{2}(\Rd)$ is defined as
\begin{equation} \label{def_eq_was}
\W_{2}(p_{1},p_{2}) \coloneqq \sqrt{\inf_{q \in \Cpl(p_{1},p_{2})} \int \vert x_{1} - x_{2}\vert^{2} \, q(dx_{1},dx_{2})}.
\end{equation}
In these terms, the value of \eqref{MBMBB} can be expressed as
\[
MT(\mu,\nu) 
= \inf_{\pi \in \MT(\mu,\nu)}  \int \W_{2}^{2}(\pi_{x},\gamma) \, \mu(dx) 
- \int \vert x \vert^{2} \, d\mu(x).
\]

\subsection{Structure of optimizers}

From \cite[Theorem 6.6]{BBST23} we recall the following characterization of the dual optimizer $\hat{\psi}$ of \eqref{WeakDual} and of the primal optimizer $\hat{\pi} \in \MT(\mu,\nu)$ of \eqref{eq_primal}.

\begin{lemma} \label{theo_du_op_b_m} Let $\mu, \nu \in \PP_{2}(\Rd)$ with $\mu \lc \nu$. Suppose that a Bass martingale $(\hat{M}_{t})_{0 \leqslant t \leqslant 1}$ from $\mu$ to $\nu$ with Bass measure $\hat{\alpha} \in \P(\R^{d})$ and associated convex function $\hat{v}$ exists. Then the Legendre transform $\hat{v}^{\ast}$ is equal to the dual optimizer $\hat{\psi}$ of \eqref{WeakDual} and $\Law(\hat{M}_{0},\hat{M}_{1})$ is equal to the primal optimizer $\hat{\pi}$ of \eqref{eq_primal}. Furthermore, we have $\hat{\alpha} = \nabla \hat{\varphi}(\mu)$, where
\begin{equation} \label{eq_du_op_b_m_xi_fin}
\nabla \hat{\varphi}(x) 
= (\nabla \hat{v} \ast \gamma)^{-1}(x)
= \nabla(\hat{v} \ast \gamma)^{\ast}(x),
\end{equation}
for $\mu$-a.e.\ $x \in \Rd$, and 
\begin{equation} \label{eq_id_pisbmx_v_xi}
\hat{\pi}_{x} = \Law(\hat{M}_{1} \, \vert \, \hat{M}_{0} = x) 
= \nabla \hat{v}(\gamma_{\nabla \hat{\varphi}(x)}),
\end{equation}
where $\gamma_{\nabla \hat{\varphi}(x)}$ denotes the $d$-dimensional Gaussian distribution with barycenter $\nabla \hat{\varphi}(x)$ and covariance matrix $I_{d}$.
\end{lemma}

We set $\hat{u} \coloneqq \hat{v} \ast \gamma$, so that $\nabla \hat{u}(\hat{\alpha}) = \mu$. Recalling \eqref{eq_def_id_bm}, we summarize the relationships between the optimizers in the following diagram:
\[
\begin{tikzcd}[row sep=huge, column sep = huge]
\hat{\alpha} \ast \gamma \arrow[shift right, swap]{r}{\nabla \hat{v}} & \nu \arrow[shift right, swap]{l}{\nabla \hat{\psi}} \\
\hat{\alpha} \arrow[shift right, swap]{r}{\nabla \hat{u}}  \arrow{u}{\ast} & \mu 
\arrow[swap]{u}{\hat{\pi}_{\cdot}} \arrow[shift right, swap]{l}{\nabla \hat{\varphi}} 
\end{tikzcd}
\]

\smallskip

Finally, we prove the equivalence between the identities \eqref{eq_def_id_bm} and the existence of a Bass martingale from $\mu$ to $\nu$.

\begin{lemma} \label{lem_eq_co_bm_13} Let $\mu, \nu \in \PP_{2}(\Rd)$ with $\mu \lc \nu$. There is a Bass martingale $\hat{M}$ with Bass measure $\hat{\alpha} \in \PP(\Rd)$ from $\mu = \Law(\hat{M}_{0})$ to $\nu = \Law(\hat{M}_{1})$ if and only if there is a convex function $\hat{v} \colon \Rd \rightarrow \R$ satisfying the identities 
\begin{equation} \label{eq_def_id_bm_rep}
(\nabla \hat{v} \ast \gamma)(\hat{\alpha}) = \mu
\qquad \textnormal{ and } \qquad 
\nabla \hat{v}(\hat{\alpha} \ast \gamma) = \nu.
\end{equation}
Moreover, the Bass martingale $\hat{M}$ can be expressed as 
\begin{equation} \label{eq_def_id_bm_sv_rep}
\hat{M}_{t} = \nabla \hat{v}_{t}(B_{t}), \qquad 0 \leqslant t \leqslant 1.
\end{equation}
\begin{proof} Let $\hat{M}$ be a Bass martingale in the sense of Definition \ref{def:BassMarti_intro}. We first prove \eqref{eq_def_id_bm_sv_rep}. Let $A \subseteq \Rd$ be a Borel set. We have to show that
\begin{equation} \label{lem_eq_co_bm_13_i}
\mathbb{E}\big[\nabla \hat{v}(B_{1}) \, \boldsymbol{1}_{\{B_{t} \in A \}}\big] 
= \mathbb{E}\big[(\nabla \hat{v} \ast \gamma^{1-t})(B_{t}) \, \boldsymbol{1}_{\{B_{t} \in A \}}\big].
\end{equation}
Denote by $\varphi_{t}(x,y)$ the Gaussian kernel, for $t \in (0,1]$ and $x,y \in \Rd$. Then the left-hand side of \eqref{lem_eq_co_bm_13_i} can be expressed as
\[
\int \hat{\alpha}(dx_{0}) \int_{A} \varphi_{t}(x_{0},dx_{t}) \int \nabla \hat{v}(x_{1}) \, \varphi_{1-t}(x_{t},dx_{1}),
\]
while the right-hand side is equal to
\[
\int \hat{\alpha}(dx_{0}) \int_{A} (\nabla \hat{v} \ast \gamma^{1-t})(x_{t}) \, \varphi_{t}(x_{0},dx_{t}).
\]
Now we see that \eqref{lem_eq_co_bm_13_i} follows from 
\[
\int \nabla \hat{v}(x_{1}) \, \varphi_{1-t}(x_{t},dx_{1}) 
= \int \nabla \hat{v}(x_{1}) \, \gamma_{x_{t}}^{1-t}(dx_{1}) 
= (\nabla \hat{v} \ast \gamma^{1-t})(x_{t}),
\]
where $\gamma_{x_{t}}^{1-t}$ denotes the $d$-dimensional Gaussian distribution with barycenter $x_{t}$ and covariance matrix $(1-t)I_{d}$. This completes the proof of \eqref{eq_def_id_bm_sv_rep}.

\smallskip

In particular, at times $t=0$ and $t=1$ we obtain from \eqref{eq_def_id_bm_sv_rep} that $\hat{M}_{0} = (\nabla \hat{v} \ast \gamma)(B_{0})$ and $\hat{M}_{1} = \nabla \hat{v}(B_{1})$, respectively. If $\hat{M}$ is a Bass martingale from $\mu = \Law(\hat{M}_{0})$ to $\nu = \Law(\hat{M}_{1})$, this readily gives \eqref{eq_def_id_bm_rep}

\medskip

Conversely, suppose that $\mu,\nu,\hat{\alpha},\hat{v}$ satisfy the identities \eqref{eq_def_id_bm_rep}. Let $(B_{t})_{0 \leqslant t \leqslant 1}$ be Brownian motion on $\Rd$ with $\Law(B_{0}) = \hat{\alpha}$. We then define a process $(\hat{M}_{t})_{0 \leqslant t \leqslant 1}$ by \eqref{def:BassMarti_intro_eq}. In light of the previous argument, $\hat{M}$ is characterized by \eqref{eq_def_id_bm_sv_rep}. Since by assumption the identities \eqref{eq_def_id_bm_rep} are satisfied, we see that $\Law(\hat{M}_{0}) = \mu$ and $\Law(\hat{M}_{1}) = \nu$. Thus $\hat{M}$ is indeed a Bass martingale from $\mu$ to $\nu$.
\end{proof}
\end{lemma}

\section{A variational characterization of Bass measures} \label{sec2avcob}

\noindent Throughout this section we fix $\mu, \nu \in \PP_{2}(\Rd)$ with $\mu \lc \nu$ and provide the proof of Theorem \ref{prop_alpha_vari_mc}. This is done in several steps.

\begin{lemma} \label{prop_alpha_vari_mc_step1} We have the weak duality 
\begin{equation} \label{eq:variational_alpha_mc_step1}
P(\mu,\nu) 
\leqslant \inf_{\alpha \in \P_{2}(\R^{d})} \mathcal{V}(\alpha).
\end{equation}
\begin{proof} Let $\alpha \in \P_{2}(\R^{d})$ be arbitrary. By Brenier's theorem \cite[Theorem 2.12]{Vi03} there is a convex function $v$ such that $\nabla v(\alpha \ast \gamma) = \nu$. Hence from the Kantorovich duality \cite[Theorem 5.10]{Vi09} it follows that
\[
\MCov(\alpha \ast \gamma,\nu) 
= \int v \, d(\alpha \ast \gamma) + \int v^{\ast} \, d\nu 
= \int ( v \ast \gamma) \, d\alpha  + \int v^{\ast} \, d\nu.
\]
Since $v \ast \gamma$ is convex, applying once more the Kantorovich duality yields
\begin{align*}
\MCov(\alpha \ast \gamma,\nu) 
&= \int v^{\ast} \, d\nu - \int ( v \ast \gamma)^{\ast} \, d\mu 
+ \int ( v \ast \gamma) \, d\alpha + \int ( v \ast \gamma)^{\ast} \, d\mu  \\
&\geqslant \int v^{\ast} \, d\nu - \int ( v \ast \gamma)^{\ast} \, d\mu + \MCov(\alpha,\mu).
\end{align*}
Finally, from Theorem \ref{theorem_new_duality} we deduce that
\begin{align*}
\MCov(\alpha \ast \gamma,\nu) 
&\geqslant \inf_{\textnormal{$\psi$ convex}} \Big( \int \psi \, d\nu - \int (\psi^{\ast} \ast \gamma)^{\ast} \, d \mu \Big) + \MCov(\alpha,\mu) \\
&= P(\mu,\nu) + \MCov(\alpha,\mu),
\end{align*}
which gives the inequality \eqref{eq:variational_alpha_mc_step1}. We remark that it is immaterial whether in \eqref{WeakDual} we optimize over convex functions $\psi$ which are elements of $L^{1}(\nu)$ or which are just $\mu$-a.s.\ finite, see \cite[Section 4]{BBST23}. 
\end{proof}
\end{lemma}

\begin{lemma} \label{prop_alpha_vari_mc_step2} Suppose that there exists a Bass martingale from $\mu$ to $\nu$ with Bass measure $\hat{\alpha} \in \P_{2}(\R^{d})$. Then the right-hand side of \eqref{eq:variational_alpha_mc_step1} is attained by $\hat{\alpha}$ and is equal to
\begin{equation} \label{prop_alpha_vari_mc_step2_eq}
\mathcal{V}(\hat{\alpha})
= \int \MCov(\hat{\pi}_{x},\gamma) \, \mu(dx),
\end{equation}
where $\hat{\pi} \in \MT(\mu,\nu)$ is the optimizer of \eqref{eq_primal}.
\begin{proof} By assumption there exists a Bass martingale from $\mu$ to $\nu$, with Bass measure $\hat{\alpha} \in \P_{2}(\R^{d})$ and associated convex function $\hat{v}$ satisfying (recall Lemma \ref{lem_eq_co_bm_13}) the identities \eqref{eq_def_id_bm_rep}. According to Lemma \ref{theo_du_op_b_m}, we have that $\hat{\alpha} = \nabla \hat{\varphi}(\mu)$ and 
\[
\hat{\pi}_{x} 
= \nabla \hat{v}(\gamma_{\nabla \hat{\varphi}(x)}),
\]
for $\mu$-a.e.\ $x \in \Rd$. Applying Brenier's theorem, we deduce that
\begin{align*}
&\int \MCov(\hat{\pi}_{x},\gamma) \, \mu(dx) 
= \int \int \big\langle \nabla \hat{v}\big(\nabla\hat{\varphi}(x) + z \big), z \big\rangle \, \gamma(dz) \, \mu(dx) \\
& \qquad = \int \int \Big( \big\langle \nabla \hat{v}\big(\nabla\hat{\varphi}(x) + z \big), \nabla\hat{\varphi}(x) + z \big\rangle - \big\langle \nabla \hat{v}\big(\nabla\hat{\varphi}(x) + z \big), \nabla\hat{\varphi}(x) \big\rangle \Big) \, \gamma(dz) \, \mu(dx)   \\
& \qquad = \int \int \big( \langle \nabla \hat{v}(a + z ), a + z \rangle - \langle \nabla \hat{v}(a + z ), a \rangle \big) \, \gamma(dz) \, \hat{\alpha}(da) \\
&\qquad = \int \langle \nabla \hat{v}, \operatorname{Id} \rangle \, d(\hat{\alpha} \ast \gamma) - \int \langle  (\nabla\hat{v} \ast \gamma), \operatorname{Id} \rangle \, d\hat{\alpha} \\
&\qquad = \MCov(\hat{\alpha} \ast \gamma,\nu) - \MCov(\hat{\alpha},\mu) 
= \mathcal{V}(\hat{\alpha}),
\end{align*}
which shows \eqref{prop_alpha_vari_mc_step2_eq}. Together with the weak duality \eqref{eq:variational_alpha_mc_step1} of Lemma \ref{prop_alpha_vari_mc_step1} above, and recalling from Subsection \ref{subs_smot} that the right-hand side of \eqref{prop_alpha_vari_mc_step2_eq} is equal to $\tilde{P}(\mu,\nu) = P(\mu,\nu)$, we conclude the assertion of Lemma \ref{prop_alpha_vari_mc_step2}.
\end{proof}
\end{lemma}

\begin{lemma} \label{prop_alpha_vari_mc_step3} We have the duality result
\begin{equation} \label{eq:variational_alpha_mc_step3}
P(\mu,\nu)
= \inf_{\alpha \in \P_{2}(\R^{d})} \mathcal{V}(\alpha).
\end{equation}
\begin{proof} For $\varepsilon > 0$ we define $\mu^{\varepsilon} \coloneqq \mu\ast \gamma^{\varepsilon}$ and $\nu^{\varepsilon} \coloneqq \nu \ast \gamma^{2\varepsilon}$. Then $\mu^{\varepsilon} \lc \nu^{\varepsilon}$ and the pair $(\mu^{\varepsilon} ,\nu^{\varepsilon})$ is irreducible. Hence by Theorem \ref{MainTheorem} there is a Bass martingale from $\mu^{\varepsilon}$ to $\nu^{\varepsilon}$, so that by Lemma \ref{prop_alpha_vari_mc_step2} we have
\begin{equation} \label{eq:variational_alpha_mc_e}
\sup_{\pi \in \MT(\mu^{\varepsilon},\nu^{\varepsilon})} \int \MCov(\pi_{x},\gamma) \, \mu^{\varepsilon}(dx) 
= \inf_{\alpha \in \P_{2}(\R^{d})} \Big( \MCov(\alpha \ast \gamma,\nu^{\varepsilon}) - \MCov(\alpha,\mu^{\varepsilon}) \Big).  
\end{equation}
By weak optimal transport arguments (see \cite[Theorem 2.3]{BeJoMaPa21b}) we know
\[
\limsup_{\varepsilon \rightarrow 0} \sup_{\pi \in \MT(\mu^{\varepsilon},\nu^{\varepsilon})} 
\int \MCov(\pi_{x},\gamma) \, \mu^{\varepsilon}(dx)
\leqslant \sup_{\pi \in \MT(\mu,\nu)} \int \MCov(\pi_{x},\gamma) \, \mu(dx).
\]
Therefore, if we can show that the right-hand side of \eqref{eq:variational_alpha_mc_e} converges to the right-hand side of \eqref{eq:variational_alpha_mc_step3}, we will obtain the inequality
\[
P(\mu,\nu) 
\geqslant \inf_{\alpha \in \P_{2}(\R^{d})} \mathcal{V}(\alpha)
= \inf_{\alpha \in \P_{2}(\R^{d})} \Big( \MCov(\alpha \ast \gamma,\nu) - \MCov(\alpha,\mu) \Big),
\]
which, together with the weak duality of Lemma \ref{prop_alpha_vari_mc_step1}, establishes \eqref{eq:variational_alpha_mc_step3}. But this follows easily from 
\[
\vert \MCov(\alpha,\mu^{\varepsilon}) -\MCov(\alpha,\mu) \vert
\leqslant c_{1} \varepsilon + \tfrac{1}{2} \vert \W_{2}^{2}(\alpha,\mu^{\varepsilon}) - \W_{2}^{2}(\alpha,\mu) \vert 
\leqslant c_{2}(\varepsilon+\varepsilon^{2})
\]
and a similar estimate for $\vert\MCov(\alpha \ast \gamma,\nu^{\varepsilon}) -\MCov(\alpha\ast\gamma,\nu)\vert$.
\end{proof}
\end{lemma}

\begin{lemma} \label{prop_alpha_vari_mc_step4} Suppose that the right-hand side of \eqref{eq:variational_alpha_mc_step3} is attained by $\hat{\alpha} \in \P_{2}(\R^{d})$. Then there exists a Bass martingale from $\mu$ to $\nu$ with Bass measure $\hat{\alpha}$.
\begin{proof} By Brenier's theorem there is a convex function $\hat{v}$ such that $\nabla \hat{v}(\hat{\alpha} \ast \gamma) = \nu$. According to Lemma \ref{lem_eq_co_bm_13}, for the existence of a Bass martingale from $\mu$ to $\nu$, it remains to show the first equality in \eqref{eq_def_id_bm_rep}, i.e.,
\begin{equation} \label{prop_alpha_vari_mc_step4_i} 
(\nabla \hat{v} \ast \gamma)(\hat{\alpha}) = \mu.
\end{equation}
Let $\hat{Z}$ and $X$ be random variables with laws $\hat{\alpha}$ and $\mu$, respectively, such that
\begin{equation} \label{prop_alpha_vari_mc_step4_eq_aa} 
\MCov(\hat{\alpha},\mu) = \E\big[ \langle \hat{Z},X \rangle \big].
\end{equation}
Denote by $\hat{q}(dz,dx)$ the law of the coupling $(\hat{Z},X)$. Let $\boldsymbol{w} \colon \Rd \times \Rd \rightarrow \Rd$ be a smooth function with compact support and define probability measures $(\alpha_{u})_{u \in \R} \subseteq \P_{2}(\Rd)$ by
\begin{equation} \label{prop_alpha_vari_mc_step4_eq_ab} 
\int f \, d\alpha_{u} \coloneqq \int \int f\big(z + u\boldsymbol{w}(z,x)\big) \, q(dz,dx), \qquad f \in C_{b}(\Rd).
\end{equation}
We claim that
\begin{equation} \label{prop_alpha_vari_mc_step4_cl_1}
\liminf_{u \rightarrow 0} \tfrac{1}{u} \Big( \MCov(\alpha_{u},\mu) - \MCov(\hat{\alpha},\mu) \Big) 
\geqslant \E\big[ \big\langle \boldsymbol{w}(\hat{Z},X),X \big\rangle \big]
\end{equation}
and
\begin{equation} \label{prop_alpha_vari_mc_step4_cl_2}
\lim_{u \rightarrow 0} \tfrac{1}{u} \Big( \MCov(\alpha_{u} \ast \gamma,\nu) - \MCov(\hat{\alpha} \ast \gamma,\nu) \Big) 
= \E\big[ \big\langle \boldsymbol{w}(\hat{Z},X), (\nabla \hat{v} \ast \gamma)(\hat{Z}) \big\rangle \big].
\end{equation}
Using the optimality of $\hat{\alpha} \in \P_{2}(\Rd)$ for the right-hand side of \eqref{eq:variational_alpha_mc_step3} and admitting the two claims \eqref{prop_alpha_vari_mc_step4_cl_1}, \eqref{prop_alpha_vari_mc_step4_cl_2}, we deduce that 
\begin{align*} 
0 &\leqslant 
\liminf_{u \rightarrow 0} \tfrac{1}{u} \bigg( 
\Big( \MCov(\alpha_{u} \ast \gamma,\nu) - \MCov(\hat{\alpha} \ast \gamma,\nu) \Big)
- \Big( \MCov(\alpha_{u},\mu) - \MCov(\hat{\alpha},\mu) \Big) \bigg) \\
&\leqslant \E\big[ \big\langle \boldsymbol{w}(\hat{Z},X) , (\nabla \hat{v} \ast \gamma)(\hat{Z}) - X \big\rangle \big]. 
\end{align*}
Since $\boldsymbol{w}$ was arbitrary, it follows that the random variable $(\nabla \hat{v} \ast \gamma)(\hat{Z})$ has the same law as $X$, which readily gives \eqref{prop_alpha_vari_mc_step4_i}.

\smallskip

We now turn to the proof of the claim \eqref{prop_alpha_vari_mc_step4_cl_1}. By the definition of $\alpha_{u}$ in \eqref{prop_alpha_vari_mc_step4_eq_ab}, the random variable $Z_{u} \coloneqq \hat{Z} + u \boldsymbol{w}(\hat{Z},X)$ has law $\alpha_{u}$. Consequently,
\begin{equation} \label{prop_alpha_vari_mc_step4_eq_ac} 
\MCov(\alpha_{u},\mu) \geqslant \E\big[ \langle Z_{u},X  \rangle \big].
\end{equation}
Combining \eqref{prop_alpha_vari_mc_step4_eq_aa} and \eqref{prop_alpha_vari_mc_step4_eq_ac} yields \eqref{prop_alpha_vari_mc_step4_cl_1}.

\smallskip

It remains to show the claim \eqref{prop_alpha_vari_mc_step4_cl_2}. By analogy with the proof of \eqref{prop_alpha_vari_mc_step4_cl_1}, we obtain the inequality ``$\geqslant$'' in \eqref{prop_alpha_vari_mc_step4_cl_2}.
%Take a standard Gaussian random vector $\Gamma$ on $\Rd$, independent of $\hat{Z}$ as well as of %$X$. Then the random variables $\hat{Z} + \Gamma$ and $Z_{u} + \Gamma$ have laws $\hat{\alpha} %\ast \gamma$ and $\alpha_{u} \ast \gamma$, respectively. Let $Y \coloneqq \nabla \hat{v}(\hat{Z} %+ \Gamma)$. Then $Y$ has law $\nu$ and 
%\begin{equation} \label{prop_alpha_vari_mc_step4_eq_ad} 
%\MCov(\hat{\alpha} \ast \gamma,\nu) = \E\big[ \langle \hat{Z} + \Gamma,Y  \rangle \big].
%\end{equation}
%Furthermore, we have
%\begin{equation} \label{prop_alpha_vari_mc_step4_eq_ae} 
%\MCov(\alpha_{u} \ast \gamma,\nu) \geqslant \E\big[ \langle Z_{u} + \Gamma,Y  \rangle \big].
%\end{equation}
%By analogy with the proof of \eqref{prop_alpha_vari_mc_step4_cl_1}, combining %\eqref{prop_alpha_vari_mc_step4_eq_ad} and \eqref{prop_alpha_vari_mc_step4_eq_ae} gives 
%\begin{align*} 
%\liminf_{u \rightarrow 0} \tfrac{1}{u} \Big( \MCov(\alpha_{u} \ast \gamma,\nu) - %\MCov(\hat{\alpha} \ast \gamma,\nu) \Big) 
%&\geqslant \E\big[ \big\langle \boldsymbol{w}(\hat{Z},X), \nabla \hat{v}(\hat{Z} + \Gamma) %\big\rangle \big] \\
%&= \E\Big[ \E\big[ \big\langle \boldsymbol{w}(\hat{Z},X), \nabla \hat{v}(\hat{Z} + \Gamma) %\big\rangle \, \big \vert \, \hat{Z},X\big]\Big] \\
%&=\E\big[ \big\langle \boldsymbol{w}(\hat{Z},X), (\nabla \hat{v} \ast \gamma)(\hat{Z}) %\big\rangle \big].
%\end{align*}
For the reverse inequality, we note that by the Kantorovich duality we have
\begin{align*}
\MCov(\alpha_{u} \ast \gamma,\nu) 
&= \inf_{\textnormal{$v$ convex}} \Big( \int v \, d(\alpha_{u} \ast \gamma) - \int v^{\ast} \, d \nu \Big) \\
&\leqslant \int \hat{v} \, d(\alpha_{u} \ast \gamma) - \int \hat{v}^{\ast} \, d \nu \\
&= \int (\hat{v} \ast \gamma) \, d\alpha_{u} - \int \hat{v}^{\ast} \, d \nu
\end{align*}
and
\[
\MCov(\hat{\alpha} \ast \gamma,\nu) 
= \int \hat{v} \, d(\alpha \ast \gamma) + \int \hat{v}^{\ast} \, d\nu 
= \int ( \hat{v} \ast \gamma) \, d\hat{\alpha}  + \int \hat{v}^{\ast} \, d\nu.
\]
Therefore
\begin{align*}
\MCov(\alpha_{u} \ast \gamma,\nu) - \MCov(\hat{\alpha} \ast \gamma,\nu)    
&\leqslant \int (\hat{v} \ast \gamma) \, d\alpha_{u} - \int ( \hat{v} \ast \gamma) \, d\hat{\alpha} \\
&= \E\big[(\hat{v} \ast \gamma)\big(\hat{Z} + u \boldsymbol{w}(\hat{Z},X)\big) - (\hat{v} \ast \gamma)(\hat{Z})\big]
\end{align*}
Using the convexity of the function $\hat{v} \ast \gamma$, we deduce that
\[
\tfrac{1}{u} \Big( \MCov(\alpha_{u} \ast \gamma,\nu) - \MCov(\hat{\alpha} \ast \gamma,\nu) \Big) 
\leqslant \E\big[ \big\langle \boldsymbol{w}(\hat{Z},X), (\nabla \hat{v} \ast \gamma)\big(\hat{Z}+ u\boldsymbol{w}(\hat{Z},X)\big) \big\rangle \big].
\]
Now observe that the expectation on the right-hand side of the above inequality is equal to the expectation of the random variable
\[
Y_{u} \coloneqq \Big\langle \boldsymbol{w}(\hat{Z},X) \, , \nabla \hat{v}(\hat{Z} + \Gamma) \exp\big(u \big\langle \Gamma ,  \boldsymbol{w}(\hat{Z},X) \big\rangle - \tfrac{u^{2}}{2} \vert \boldsymbol{w}(\hat{Z},X) \vert^{2}\big)\Big\rangle,
\]
where $\Gamma$ is a standard Gaussian random vector on $\Rd$, independent of $\hat{Z}$ as well as of $X$. Clearly by continuity
\[
\lim_{u \rightarrow 0} Y_{u} = 
\big\langle \boldsymbol{w}(\hat{Z},X) , \nabla \hat{v}(\hat{Z} + \Gamma) \big\rangle, \qquad \mathbb{P}\textnormal{-a.s.}
\]
As $\boldsymbol{w}$ is smooth with compact support, for $\delta > 0$ we can find constants $c_{1},c_{2}$ such that
\[
\forall u \in [-\delta,\delta] \colon \quad
\vert Y_{u} \vert \leqslant c_{1} \, \vert \nabla \hat{v}(\hat{Z} + \Gamma) \vert \, \mathrm{e}^{c_{2} \vert \Gamma \vert}.
\]
By the Cauchy--Schwarz inequality and since $\nabla \hat{v}(\hat{\alpha} \ast \gamma) = \nu \in \PP_{2}(\Rd)$, we have the bound
\[
\mathbb{E}\big[ \vert \nabla \hat{v}(\hat{Z} + \Gamma) \vert \, \mathrm{e}^{\vert \Gamma \vert} \big] 
%\leqslant \sqrt{\mathbb{E}\big[ \vert \nabla \hat{v}(\hat{Z} + \Gamma) \vert^{2} \big]} 
%\, \sqrt{\mathbb{E}\big[  \mathrm{e}^{2 \vert \Gamma \vert} \big]} 
\leqslant \sqrt{\int \vert y \vert^{2} \, d\nu(y)} \ \sqrt{\mathbb{E}\big[  \mathrm{e}^{2 \vert \Gamma \vert} \big]} < + \infty.
\]
Therefore we can apply the dominated convergence theorem and conclude that
\[
\limsup_{u \rightarrow 0} \tfrac{1}{u} \Big( \MCov(\alpha_{u} \ast \gamma,\nu) - \MCov(\hat{\alpha} \ast \gamma,\nu) \Big) 
\leqslant \E\big[ \big\langle \boldsymbol{w}(\hat{Z},X), (\nabla \hat{v} \ast \gamma)(\hat{Z}) \big\rangle \big],
\]
which completes the proof of the claim \eqref{prop_alpha_vari_mc_step4_cl_2}.
\end{proof}
\end{lemma}

\begin{proof}[Proof of Theorem \ref{prop_alpha_vari_mc}] The assertion of the theorem follows from Lemmas \ref{prop_alpha_vari_mc_step2} -- \ref{prop_alpha_vari_mc_step4}.
\end{proof}

The reader has certainly noticed that the proof of Lemma \ref{prop_alpha_vari_mc_step2} was given in an analytic style while the proof of Lemma \ref{prop_alpha_vari_mc_step4} was given in a more probabilistic language. In the remainder of this section we give an alternative probabilistic proof of Lemma \ref{prop_alpha_vari_mc_step2} and sketch how to translate the proof of Lemma \ref{prop_alpha_vari_mc_step4} into a more analytic language. 

\smallskip

The following probabilistic proof of Lemma \ref{prop_alpha_vari_mc_step2} does not require the duality results developed in \cite{BBST23}, but only relies on the definition of Bass martingales.

%\begin{lemma} \label{prop_alpha_vari_mc_step2_alt} Suppose that a Bass martingale from $\mu$ to %$\nu$ with Bass measure $\hat{\alpha} \in \P_{2}(\R^{d})$ exists. Then 
%\begin{equation} \label{eq:variational_alpha_mc_step1_alt}
%\inf_{\alpha \in \P_{2}(\R^{d})} \Big( \MCov(\alpha \ast \gamma,\nu) - \MCov(\alpha,\mu) \Big).
%\end{equation}
%is attained by $\hat{\alpha}$.
%\end{lemma}

\begin{proof}[Probabilistic proof of Lemma \ref{prop_alpha_vari_mc_step2}] By assumption there exists a Bass martingale from $\mu$ to $\nu$, with Bass measure $\hat{\alpha} \in \P_{2}(\R^{d})$ and associated convex function $\hat{v}$ satisfying (recall Lemma \ref{lem_eq_co_bm_13}) the identities \eqref{eq_def_id_bm_rep}. Let $\alpha \in \P_{2}(\R^{d})$ be arbitrary. We have to show that
\begin{equation} \label{eq:variational_alpha_mc_step1_alt_pro}
\begin{aligned} 
\mathcal{V}(\hat{\alpha}) = & \, \MCov(\hat{\alpha} \ast \gamma,\nu) - \MCov(\hat{\alpha},\mu) \leqslant \\
\leqslant & \, \MCov(\alpha \ast \gamma,\nu) - \MCov(\alpha,\mu) = \mathcal{V}(\alpha).
\end{aligned}
\end{equation}
Take a random variable $\hat{Z}$ with law $\hat{\alpha}$ and define
\begin{equation} \label{eq_prop_alpha_vari_mc_prop_ii_pre_bb}
X \coloneqq (\nabla \hat{v} \ast \gamma)(\hat{Z}).
\end{equation}
By Brenier's theorem the coupling $(\hat{Z},X)$ is optimal and according to \eqref{eq_def_id_bm_rep} the random variable $X$ has law $\mu$. Now choose a random variable $Z$ with law $\alpha$ such that the coupling $(Z,X)$ is optimal with respect to the maximal covariance (equivalently, with respect to the quadratic Wasserstein distance). Clearly
\begin{equation} \label{eq_prop_alpha_vari_mc_prop_ii_pre}
\MCov(\alpha,\mu) - \MCov(\hat{\alpha},\mu) 
= \E\big[\langle Z -  \hat{Z},X \rangle\big].
\end{equation}
Take a standard Gaussian random vector $\Gamma$ on $\Rd$, independent of $Z$ as well as of $\hat{Z}$. The random variables $\hat{Z} + \Gamma$ and 
\begin{equation} \label{eq_prop_alpha_vari_mc_prop_ii_pre_aa}
Y \coloneqq \nabla \hat{v}(\hat{Z} + \Gamma)    
\end{equation}
have laws $\hat{\alpha} \ast \gamma$ and $\nu$, respectively. As by Brenier's theorem the coupling $(\hat{Z} + \Gamma,Y)$ is optimal, we have
\begin{equation} \label{eq_prop_alpha_vari_mc_prop_iii_pre}
\MCov(\hat{\alpha} \ast \gamma,\nu) = \E\big[\langle \hat{Z} + \Gamma,Y \rangle\big].
\end{equation}
Since the random variable $Z + \Gamma$ has law $\alpha \ast \gamma$ we conclude that $(Z+\Gamma,Y)$ is some coupling between $\alpha \ast \gamma$ and $\nu$, i.e.,
\begin{equation} \label{eq_prop_alpha_vari_mc_prop_iv_pre}
\MCov(\alpha \ast \gamma,\nu) \geqslant \E\big[ \langle Z+\Gamma, Y\rangle \big].
\end{equation}
From \eqref{eq_prop_alpha_vari_mc_prop_ii_pre} -- \eqref{eq_prop_alpha_vari_mc_prop_iv_pre} we obtain the inequality 
\[
\MCov(\alpha \ast \gamma,\nu) - \MCov(\hat{\alpha} \ast \gamma,\nu)
- \MCov(\alpha,\mu) + \MCov(\hat{\alpha},\mu) \geqslant \E\big[\langle Z -  \hat{Z},Y - X \rangle\big].    
\]
Therefore, in order to establish the inequality \eqref{eq:variational_alpha_mc_step1_alt_pro}, it remains to show that
\begin{equation} \label{eq_prop_alpha_vari_mc_prop_v_pre}
\E\big[\langle Z -  \hat{Z},Y - X \rangle\big] = 0.
\end{equation}
For that purpose, we condition $Y - X$ on the random variables $Z$ as well as $\hat{Z}$, so that by \eqref{eq_prop_alpha_vari_mc_prop_ii_pre_bb} and \eqref{eq_prop_alpha_vari_mc_prop_ii_pre_aa} we obtain
\[
\E[ Y - X \, \vert \, Z, \hat{Z} ] = 0,
\]
which implies \eqref{eq_prop_alpha_vari_mc_prop_v_pre}.
\end{proof}

We finally give an alternative heuristic argument for Lemma \ref{prop_alpha_vari_mc_step4}, which is based on differentiating the maximal covariance along a continuity equation.

\begin{proof}[Alternative heuristic proof of Lemma \ref{prop_alpha_vari_mc_step4}] Suppose that the right-hand side of \eqref{eq:variational_alpha_mc_step3} is attained by $\hat{\alpha} \in \P_{2}(\R^{d})$. We want to show that there exists a Bass martingale from $\mu$ to $\nu$ with Bass measure $\hat{\alpha}$. The idea is to perturb $\hat{\alpha}$ along a continuity equation
\[
\partial_{t} \alpha_{t} + \operatorname{div}(\boldsymbol{v}_{t} \alpha_{t}) = 0, \qquad t \in (-h,h),
\]
with $h > 0$, $\alpha_{0} \coloneqq \hat{\alpha}$, and where $\boldsymbol{v}_{t}$ is a velocity field. Observe that
\begin{align*}
\partial_{t} \vert_{t=0} \, \mathcal{V}(\alpha_{t})
&= \partial_{t} \vert_{t=0} \, \Big( \MCov(\alpha_{t} \ast \gamma,\nu) - \MCov(\alpha_{t},\mu) \Big) \\
&= \partial_{t} \vert_{t=0} \int \hat{v} \, d(\alpha_{t} \ast \gamma) 
- \partial_{t} \vert_{t=0} \int \hat{u} \, d\alpha_{t},
\end{align*}
where $\nabla \hat{v}(\hat{\alpha}\ast\gamma)=\nu$ is optimal and likewise $\nabla\hat{u}(\hat{\alpha})=\mu$ is optimal. By the continuity equation we obtain
\[
\partial_{t} \vert_{t=0} \int \hat{u} \, d\alpha_{t} 
= \int \langle \nabla \hat{u} ,  \boldsymbol{v}_{0} \rangle \, d\hat{\alpha}.
\]
With similar computations we have
\[
\partial_{t} \vert_{t=0} \int \hat{v} \, d(\alpha_t\ast\gamma) 
= \int \langle \nabla \hat{v} \ast \gamma , \boldsymbol{v}_{0} \rangle \, d\hat{\alpha} . 
\]
As $\boldsymbol{v}_{0}$ was arbitrary and $\hat{\alpha}$ was optimal, we conclude that
\[
0 = \int \big\langle \nabla \hat{v} \ast \gamma - \nabla \hat{u} , \boldsymbol{v}_{0} \big\rangle \, d\hat{\alpha}, 
\]
so that $\nabla \hat{u}$, the optimal map from $\hat{\alpha}$ to $\mu$, is $\hat{\alpha}$-a.s.\ equal to $\nabla \hat{v} \ast\gamma$, where $\nabla \hat{v}$ is the optimal map from $\hat{\alpha}\ast\gamma$ to $\nu$. Recalling \eqref{eq_def_id_bm_rep} and Lemma \ref{lem_eq_co_bm_13}, this is precisely the structure of the Bass martingale.
\end{proof}

\section{An infinitesimal version of Theorem \ref{prop_alpha_vari_mc}} \label{sec3aivot}

\noindent We provide the proof of Theorem \ref{theo_ms2}, an infinitesimal version of Theorem \ref{prop_alpha_vari_mc}.

\begin{proof}[Proof of Theorem \ref{theo_ms2}] For a partition $\Pi = \{t_{0}, t_{1}, \ldots, t_{n} \}$ of the interval $[0,1]$ with 
\[
0 = t_{0} < t_{1} < \ldots < t_{n} = 1
\]
we denote by $\Sigma^{\Pi}$ the collection of all progressively measurable and $L^{2}$-bounded processes $(\sigma_{t}^{\Pi})_{0 \leqslant t \leqslant 1}$ such that the stochastic integral
\[
M_{t}^{\Pi} \coloneqq M_{0} + \int_{0}^{t} \sigma_{s}^{\Pi} \, dB_{s}, \qquad 0 \leqslant t \leqslant 1
\]
defines an $L^{2}$-bounded martingale with $\Law(M_{t_{k}}^{\Pi}) = \mu_{t_{k}}$, for $k = 0, \ldots, n$. We define
\begin{equation} \label{theo_ms2_eq_1} 
m^{\Pi}([t_{k-1},t_{k}]) \coloneqq \sup_{\sigma^{\Pi} \in \Sigma^{\Pi}} 
\E \Big[ \int_{t_{k-1}}^{t_{k}} \textnormal{tr}(\sigma_{s}^{\Pi}) \, ds \Big].
\end{equation}
By \cite{BaBeHuKa20}, we know that the optimizer of
\[
m^{\Pi}([0,1]) = \sup_{\sigma^{\Pi} \in \Sigma^{\Pi}} 
\E \Big[ \int_{0}^{1} \textnormal{tr}(\sigma_{s}^{\Pi}) \, ds \Big]
\]
is given, on each interval $[t_{k-1},t_{k}]$, by the stretched Brownian motion from $\mu_{t_{k-1}}$ to $\mu_{t_{k}}$. By Theorem \ref{prop_alpha_vari_mc} we have
\begin{equation} \label{theo_ms2_eq_3} 
m^{\Pi}([t_{k-1},t_{k}]) = \inf_{\alpha \in \P_{2}(\R^{d})} \Big( \MCov(\alpha \ast \gamma^{t_{k}-t_{k-1}},\mu_{t_{k}}) - \MCov(\alpha,\mu_{t_{k-1}}) \Big).
\end{equation}
For $t_{k} \in \Pi$ and a refinement $\Pi_{1}$ of $\Pi$ we have
\[
m^{\Pi}([0,t_{k}]) \geqslant m^{\Pi_{1}}([0,t_{k}]),
\]
as the process $(\sigma_{t}^{\Pi_{1}})_{0 \leqslant t \leqslant 1}$ has to satisfy more requirements than the process $(\sigma_{t}^{\Pi})_{0 \leqslant t \leqslant 1}$. We therefore may pass to a limit $m \coloneqq \lim m^{\Pi}$ along the net of finite partitions $\Pi$ of the interval $[0,1]$, which extends to a finite measure on $[0,1]$, still denoted by $m$. Clearly the measure $m$ is absolutely continuous with respect to Lebesgue measure on $[0,1]$ and we denote the corresponding density by $g(t)$, for $0 \leqslant t \leqslant 1$. We claim that, for $0 \leqslant r \leqslant u \leqslant 1$, we have
\begin{equation} \label{theo_ms2_eq_2} 
\E \Big[ \int_{r}^{u} \textnormal{tr}(\sigma_{s}) \, ds \Big]
\leqslant m([r,u]).
\end{equation}
Indeed, otherwise we could find a partition $\Pi$ with $r,u \in \Pi$, such that
\[
\E \Big[ \int_{r}^{u} \textnormal{tr}(\sigma_{s}^{\Pi}) \, ds \Big] > m^{\Pi}([r,u]),
\]
which yields a contradiction to the definition of $m^{\Pi}( \, \cdot \,)$ in \eqref{theo_ms2_eq_1}. Since \eqref{theo_ms2_eq_2} holds for all intervals $[r,u] \subseteq [0,1]$, we deduce that
\begin{equation} \label{theo_ms2_eq_4} 
\mathbb{E}\big[ \textnormal{tr}(\sigma_{t})\big]
\leqslant g(t),     
\end{equation}
for Lebesgue-a.e.\ $0 \leqslant t \leqslant 1$. From \eqref{theo_ms2_eq_3} we conclude, for Lebesgue-a.e.\ $0 \leqslant t \leqslant 1$ and for each $\alpha \in \P_{2}(\R^{d})$, the inequality
\[
g(t) \leqslant \liminf_{h \rightarrow 0} \tfrac{1}{h}
\Big( \MCov(\alpha \ast \gamma^{h},\mu_{t+h}) - \MCov(\alpha,\mu_{t}) \Big).
\]
Together with \eqref{theo_ms2_eq_4}, this finishes the proof of \eqref{eq_m_theo_ms2}.
\end{proof}

Again we provide a more analytic argument for Theorem \ref{theo_ms2}, at least on a formal level.

\begin{proof}[Alternative heuristic proof of Theorem \ref{theo_ms2}] We will use the Kantorovich duality and the Fokker--Planck equations to get a hold of $\frac{d}{dh} \MCov(\alpha \ast \gamma^{h},\mu_{t+h})$. By a change of variables we then get an equivalent expression which, when minimized, gives the left-hand side of \eqref{eq_m_theo_ms2}. We suppose here that $M$ is a strong solution of the stochastic differential equation $dM_{u} = \sigma_{u}(M_{u}) \, dB_{u}$, with $\sigma$ as benevolent as needed, so that in particular $\mu_{u}$ admits a density for each $u$.

\smallskip

We set $\rho_{h} \coloneqq \alpha \ast \gamma^{h}$, $\Sigma \coloneqq \sigma\sigma'$, and notice that for fixed $t$ we have
\begin{align*}
\partial_{h} \rho_{h}(x) &= \tfrac{1}{2} \Delta\rho_{h}(x), \qquad \rho_{0} = \alpha;\\
\partial_{h} \mu_{t+h}(x) &= \tfrac{1}{2} \sum_{i,k} \partial^{2}_{ik} \big(\Sigma_{ik}\mu_{t+h}(x)\big).
\end{align*}
By the Kantorovich duality we have
\begin{align*}
\MCov(\rho_h,\mu_{t+h}) 
&= \inf_{\textnormal{$\phi$ convex}} \int \phi \, d\rho_{h} + \int \phi^{\ast} \, d\mu_{t+h} \\ 
&= \int \phi_{\rho_{h}}^{\mu_{t+h}} \, d\rho_{h} + \int \phi^{\rho_{h}}_{\mu_{t+h}} \, d\mu_{t+h},
\end{align*}
where we denote by $\phi_{p}^{q}(\, \cdot \, )$ the convex function, which is unique up to a constant, such that $\nabla \phi_{p}^{q}(p)=q$. Using this, or more directly \cite[Theorem 23.9]{Vi09}, we have
\begin{align*}
\frac{d}{dh} \, \MCov(\rho_{h},\mu_{t+h}) 
&= \int \phi_{\rho_{h}}^{\mu_{t+h}} \, \partial_{h} \rho_{h} \, d\lambda
+ \int \phi^{\rho_{h}}_{\mu_{t+h}} \, \partial_{h} \mu_{t+h} \, d\lambda \\
&= \int \phi_{\rho_{h}}^{\mu_{t+h}} \, \tfrac{1}{2} \Delta \rho_{h} \, d\lambda
+ \int \phi^{\rho_{h}}_{\mu_{t+h}} \, \tfrac{1}{2} \sum_{i,k} \partial^{2}_{ik}(\Sigma_{ik}\mu_{t+h})\,  d\lambda \\
&= \tfrac{1}{2} \int \sum_{i,k} \partial^{2}_{i,k}(\phi_{\rho_h}^{\mu_{t+h}}) \, I_{ik} \, d\rho_{h}
+ \tfrac{1}{2} \int \sum_{i,k} \partial^{2}_{i,k}(\phi^{\rho_{h}}_{\mu_{t+h}}) \, \Sigma_{ik} \, d\mu_{t+h} \\
&= \tfrac{1}{2} \int \textnormal{tr} \big(D^{2}(\phi_{\rho_{h}}^{\mu_{t+h}})\big) \, \rho_{h} \, d\lambda
+\tfrac{1}{2} \int \textnormal{tr} \big(D^{2}(\phi^{\rho_{h}}_{\mu_{t+h}})\Sigma\big) \, \mu_{t+h} \, d\lambda,
\end{align*}
where we denote by $D$ and $D^{2}$ the Jacobian and Hessian matrix, respectively. During this proof we will use the convention that if $x \mapsto a(x) \in \R^{d}$ is an invertible vector-valued function, then $a^{-1}(x)$ denotes the inverse function, whereas if $x \mapsto A(x) \in \R^{d \times d}$ is a matrix-valued function, then $[A(x)]^{-1}$ denotes the matrix inverse of $A(x)$. Now observe that
\[
D^{2}(\phi_{\rho_{h}}^{\mu_{t+h}})(x)
=D(\nabla \phi_{\rho_{h}}^{\mu_{t+h}})(x)
=D\big((\nabla \phi^{\rho_{h}}_{\mu_{t+h}})^{-1}\big)(x)
=[D\nabla \phi^{\rho_{h}}_{\mu_{t+h}} \circ \nabla \phi_{\rho_{h}}^{\mu_{t+h}}(x)]^{-1},
\]
so that
\begin{align*}
\int \textnormal{tr}\big(D^{2}(\phi_{\rho_h}^{\mu_{t+h}})(x)\big) \, \rho_h \, d\lambda 
&= \int \textnormal{tr}\big([D\nabla \phi^{\rho_{h}}_{\mu_{t+h}} \circ \nabla \phi_{\rho_{h}}^{\mu_{t+h}}(x)]^{-1}\big) \, \rho_{h} \, d\lambda \\
&= \int \textnormal{tr}\big([D^{2} \phi^{\rho_{h}}_{\mu_{t+h}}(y)]^{-1}\big) \, \mu_{t+h} \, d\lambda.
\end{align*}
Altogether we have
\begin{align*}
\frac{d}{dh} \, \MCov(\rho_{h},\mu_{t+h})  
= \tfrac{1}{2} \int \Big(  \textnormal{tr} \big( [D^{2} \phi^{\rho_{h}}_{\mu_{t+h}}]^{-1} \big) + \textnormal{tr}\big(D^{2}(\phi^{\rho_{h}}_{\mu_{t+h}})\Sigma \big)  \Big) \, \mu_{t+h} \, d\lambda.
\end{align*}
Define now the functional on invertible, positive-semidefinite symmetric matrices 
\[
A \mapsto J(A) \coloneqq \textnormal{tr}(A^{-1})+\textnormal{tr}(A\Sigma).
\]
We remark that $J(A) \geqslant 2\textnormal{tr}(\Sigma^{1/2})$, since this is equivalent to the trivial statement 
\[
\vert A^{-1/2}-A^{1/2}\Sigma^{1/2} \vert_{\textnormal{HS}} \geqslant 0.
\]
Hence in fact the minimum of $J(\, \cdot \,)$ is attained at $A=\Sigma^{-1/2}$. We conclude that
\[
\frac{d}{dh} \Big\vert_{h=0} \, \MCov(\rho_{h},\mu_{t+h}) 
\geqslant \int \textnormal{tr}\big(\sigma_{t}(y)\big) \, \mu_{t}(dy) 
= \mathbb{E}\big[ \textnormal{tr}\big(\sigma_{t}(M_{t})\big)\big],
\]
which completes the proof of Theorem \ref{theo_ms2}.
\end{proof}

\section{Displacement convexity of the Bass functional} \label{sec4dcotbf}

\noindent We observe that the Bass functional $\alpha \mapsto \mathcal{V}(\alpha)$ provides a novel example of a convex functional with respect to the almost-Riemannian structure of the quadratic Wasserstein space $\PP_{2}$. As mentioned in \cite[Open Problem 5.17]{Vi03}, there are only few known examples of so-called displacement convex functionals (see \cite[Definition 5.10]{Vi03}, \cite[Definition 9.1.1]{AGS08}, \cite{McCgas}), and it is desirable to find new ones.

\smallskip

We shall state two versions of this result. The first one, Proposition \ref{sec_4_dis_one}, pertains to the case $d=1$, while the second one, Proposition \ref{sec_4_dis_d}, holds for general $d \in \mathbb{N}$. We also note that, contrary to the rest of this paper, we do not assume that $\mu \lc \nu$.

\begin{proposition} \label{sec_4_dis_one}
Suppose $d = 1$. Let $\mu, \nu \in \PP_{2}(\R)$. The Bass functional
\begin{equation} \label{sec_4_dis_one_aa}
\PP_{2}(\R) \ni \alpha \longmapsto \mathcal{V}(\alpha) 
= \MCov(\alpha \ast \gamma, \nu) - \MCov(\alpha,\mu)
\end{equation}
is displacement convex. Moreover, if a geodesic $(\alpha_{u})_{0 \leqslant u \leqslant 1}$ in $\PP_{2}(\R)$ is such that $\alpha_{1}$ is not a translate of $\alpha_{0}$ and if $\nu$ is not a Dirac measure, the function $u \mapsto \mathcal{V}(\alpha_{u})$ is strictly convex.
\begin{proof} We start by noting that the Bass functional $\mathcal{V}(\, \cdot \,)$ of \eqref{sec_4_dis_one_aa} can equivalently be defined in terms of the quadratic Wasserstein distance $\mathcal{W}_{2}(\, \cdot \, , \, \cdot \,)$ of \eqref{def_eq_was} rather than in terms of the maximal covariance $\MCov(\, \cdot \, , \, \cdot \,)$ of \eqref{eq_def_mcov}. Indeed, we have the identity
\[
\mathcal{V}(\alpha) 
= \MCov(\alpha \ast \gamma, \nu) - \MCov(\alpha,\mu) 
= \tfrac{1}{2} \mathcal{W}_{2}^{2}(\alpha,\mu) - \tfrac{1}{2} \mathcal{W}_{2}^{2}(\alpha \ast \gamma, \nu) + \textnormal{const},
\]
where the constant
\[
\textnormal{const} = \tfrac{d}{2} + \tfrac{1}{2} \int \vert y \vert^{2} \, d\nu(y) - \tfrac{1}{2} \int \vert x \vert^{2} \, d\mu(x) 
\]
does not depend on $\alpha$. Therefore showing the (strict) displacement convexity of the Bass functional $\mathcal{V}(\, \cdot \,)$ is equivalent to showing the (strict) displacement convexity of the functional
\begin{equation} \label{eq_func_ualph}
\mathcal{U}(\alpha) \coloneqq \mathcal{W}_{2}^{2}(\alpha,\mu) - \mathcal{W}_{2}^{2}(\alpha \ast \gamma,\nu).
\end{equation}

\medskip

Fix $\mu, \nu \in \PP_{2}(\R)$ and let $(\alpha_{u})_{0 \leqslant u \leqslant 1}$ be a geodesic in the quadratic Wasserstein space $\PP_{2}(\R)$. Using the hypothesis $d = 1$ we can choose mutually comonotone random variables $Z_{0}$, $Z_{1}$ and $X$ with laws $\alpha_{0}$, $\alpha_{1}$ and $\mu$, respectively. As $(\alpha_{u})_{0 \leqslant u \leqslant 1}$ is a geodesic, the random variable $Z_{u} \coloneqq (1-u) Z_{0} + u Z_{1}$ has law $\alpha_{u}$, for $0 \leqslant u \leqslant 1$. Also note that each $Z_{u}$ is comonotone with $X$. Let $u_{0},u \in [0,1]$. As regards the first Wasserstein distance in \eqref{eq_func_ualph}, a straightforward calculation yields
\begin{align}
&\mathcal{W}_{2}^{2}(\alpha_{u},\mu) - \mathcal{W}_{2}^{2}(\alpha_{u_{0}},\mu) = \label{sec_4_dis_one_ai1} \\
& \qquad \qquad = \E[\vert Z_{u}-X\vert^{2}] - \E[\vert Z_{u_{0}}-X\vert^{2}] \label{sec_4_dis_one_ai2} \\
& \qquad \qquad = \E[\vert Z_{u}-Z_{u_{0}}\vert^{2}] - 2 \E [ \langle Z_{u}-Z_{u_{0}} , X - Z_{u_{0}}\rangle ] \label{sec_4_dis_one_ai3} \\
& \qquad \qquad = (u-u_{0})^{2} \E[\vert Z_{1}-Z_{0}\vert^{2}] - 2 (u-u_{0}) \E [ \langle Z_{1}-Z_{0} , X - Z_{u_{0}}\rangle]. \label{sec_4_dis_one_ai4}
\end{align}
Passing to the second Wasserstein distance in \eqref{eq_func_ualph}, we take a standard Gaussian random variable $\Gamma$ on $\R$, independent of $Z_{0}$ as well as of $Z_{1}$. Next we choose a random variable $Y_{u_{0}}$ such that $(Z_{u_{0}}+\Gamma,Y_{u_{0}})$ is an optimal coupling of $(\alpha_{u_{0}} \ast \gamma,\nu)$. As $(Z_{u} \ast \Gamma,Y_{u_{0}})$ is a (typically sub-optimal) coupling of $(\alpha_{u} \ast \gamma, \nu)$, we obtain the inequality
\begin{align}
&\mathcal{W}_{2}^{2}(\alpha_{u} \ast \gamma,\nu) - \mathcal{W}_{2}^{2}(\alpha_{u_{0}} \ast \gamma,\nu) \leqslant \label{sec_4_dis_one_ai5} \\
& \qquad \qquad \leqslant \E[\vert Z_{u} + \Gamma -Y_{u_{0}}\vert^{2}] - \E[\vert Z_{u_{0}} + \Gamma -Y_{u_{0}}\vert^{2}] \label{sec_4_dis_one_ai6} \\
& \qquad \qquad = \E[\vert Z_{u}-Z_{u_{0}}\vert^{2}] - 2 \E [ \langle Z_{u}-Z_{u_{0}} , Y_{u_{0}} - (Z_{u_{0}} + \Gamma) \rangle ] \label{sec_4_dis_one_ai7} \\
& \qquad \qquad = (u-u_{0})^{2} \E[\vert Z_{1}-Z_{0}\vert^{2}] - 2 (u-u_{0}) \E [ \langle Z_{1}-Z_{0} , Y_{u_{0}} - (Z_{u_{0}} + \Gamma) \rangle]. \label{sec_4_dis_one_ai8}
\end{align}
Combining \eqref{sec_4_dis_one_ai1} -- \eqref{sec_4_dis_one_ai8}, we deduce that
\begin{align}
&\mathcal{U}(\alpha_{u}) - \mathcal{U}(\alpha_{u_{0}}) =\\
& \qquad = \Big( \mathcal{W}_{2}^{2}(\alpha_{u},\mu) - \mathcal{W}_{2}^{2}(\alpha_{u} \ast \gamma,\nu) \Big)
- \Big( \mathcal{W}_{2}^{2}(\alpha_{u_{0}},\mu) - \mathcal{W}_{2}^{2}(\alpha_{u_{0}} \ast \gamma,\nu) \Big) \geqslant \label{sec_4_dis_one_ai9} \\
& \qquad  \geqslant
2 (u-u_{0}) \E [ \langle Z_{1}-Z_{0} , Y_{u_{0}} - X - \Gamma \rangle] \label{sec_4_dis_one_ai10} \\
& \qquad  =
2 (u-u_{0}) \E [ \langle Z_{1}-Z_{0} , Y_{u_{0}} - X \rangle], \label{sec_4_dis_one_ai11} 
\end{align}
where the last equation follows from conditioning on $Z_{0},Z_{1}$. The expression in \eqref{sec_4_dis_one_ai11} defines a linear function in $u$, which lies below and touches the function 
\begin{align*}
u \longmapsto & \, \mathcal{U}(\alpha_{u}) - \mathcal{U}(\alpha_{u_{0}}) = \\
& \qquad = \Big( \mathcal{W}_{2}^{2}(\alpha_{u},\mu) - \mathcal{W}_{2}^{2}(\alpha_{u} \ast \gamma,\nu) \Big)
- \Big( \mathcal{W}_{2}^{2}(\alpha_{u_{0}},\mu) - \mathcal{W}_{2}^{2}(\alpha_{u_{0}} \ast \gamma,\nu) \Big)
\end{align*}
at the point $u = u_{0}$. This readily implies the convexity of the function 
\[
u \longmapsto \mathcal{U}(\alpha_{u}) = \mathcal{W}_{2}^{2}(\alpha_{u},\mu) - \mathcal{W}_{2}^{2}(\alpha_{u} \ast \gamma,\nu).
\]

\medskip

It remains to show the strict convexity assertion of Proposition \ref{sec_4_dis_one}. If $\alpha_{1}$ is not a translate of $\alpha_{0}$, then $\alpha_{u}$ is not a translate of $\alpha_{u_{0}}$ either, provided that $u \neq u_{0}$. As $Z_{u_{0}} + \Gamma$ is comonotone with $Y_{u_{0}}$ and $Y_{u_{0}}$ is assumed to be non-constant, we may find $y_{0} \in \R$ and $z_{0} \in \R$ such that $\mathbb{P}[Y_{u_{0}} < y_{0}] \in (0,1)$ and
\[
\{ Z_{u_{0}} + \Gamma < z_{0} \} = \{ Y_{u_{0}} < y_{0} \}.
\]
If $Z_{u} + \Gamma$ were also comonotone with $Y_{u_{0}}$, we could find $z \in \R$ such that
\[
\{ Z_{u_{0}} + \Gamma < z_{0} \} 
= \{ Y_{u_{0}} < y_{0} \}
= \{ Z_{u} + \Gamma < z \},
\]
where we have used that the law of $Z_{u} + \Gamma$ is continuous. Conditioning on $\Gamma = \zeta$ this implies that, for Lebesgue-a.e.\ $\zeta \in \R$,
\[
\{ Z_{u_{0}}  < z_{0} - \zeta \} 
= \{ Z_{u} < z - \zeta \},
\]
so that $Z_{u_{0}}$ and $Z_{u}$ are translates. This gives the desired contradiction, showing that there is a strict inequality in \eqref{sec_4_dis_one_ai5}, \eqref{sec_4_dis_one_ai6} (thus also in \eqref{sec_4_dis_one_ai9}, \eqref{sec_4_dis_one_ai10}), which implies the strict convexity assertion of Proposition \ref{sec_4_dis_one}.
\end{proof}
\end{proposition}

We now pass to the case of general $d \in \mathbb{N}$. In Proposition \ref{sec_4_dis_d} below we formulate a convexity property of the Bass functional $\mathcal{V}(\, \cdot \,)$ pertaining to the notion of \textit{generalized geodesics} as analyzed in \cite[Definition 9.2.2]{AGS08}. Recall that $(\alpha_{u})_{0 \leqslant u \leqslant 1}$ is a generalized geodesic with base $\mu$, joining $\alpha_{0}$ to $\alpha_{1}$, if there are random variables $Z_{0}$, $Z_{1}$ and $X$ with laws $\alpha_{0}$, $\alpha_{1}$ and $\mu$, respectively, such that $(Z_{0},X)$ and $(Z_{1},X)$ are optimal couplings and such that the random variable $Z_{u} \coloneqq uZ_{1} + (1-u) Z_{0}$ has law $\alpha_{u}$, for $0 \leqslant u \leqslant 1$.

\begin{proposition} \label{sec_4_dis_d} 
Let $\mu, \nu \in \PP_{2}(\R^{d})$. The Bass functional
\[
\PP_{2}(\R^{d}) \ni \alpha \longmapsto \mathcal{V}(\alpha) 
= \MCov(\alpha \ast \gamma, \nu) - \MCov(\alpha,\mu)
\]
is convex along generalized geodesics $(\alpha_{u})_{0 \leqslant u \leqslant 1}$ in $\PP_{2}(\R^{d})$ with base $\mu$.
\end{proposition}

We do not know whether the above assertion is also true along (non generalized) geodesics $(\alpha_{u})_{0 \leqslant u \leqslant 1}$ in $\PP_{2}(\R^{d})$, when $d > 1$.

\begin{proof}[Proof of Proposition \ref{sec_4_dis_d}] We follow the lines of the proof of Proposition \ref{sec_4_dis_one} and consider again the functional
\[
\mathcal{U}(\alpha) = \mathcal{W}_{2}^{2}(\alpha,\mu) - \mathcal{W}_{2}^{2}(\alpha \ast \gamma,\nu)
\]
as in \eqref{eq_func_ualph}. Let $(\alpha_{u})_{0 \leqslant u \leqslant 1}$ be a generalized geodesic with base $\mu$, joining $\alpha_{0}$ to $\alpha_{1}$. Take $Z_{0}, Z_{1}, Z_{u}, X$ as above such that $(Z_{0},X)$ and $(Z_{1},X)$ are optimal couplings and by definition $Z_{u} \sim \alpha_{u}$. Note that $(Z_{u},X)$ is an optimal coupling of $(\alpha_{u},\mu)$ by \cite[Lemma 9.2.1]{AGS08}, for $0 \leqslant u \leqslant 1$. The equalities \eqref{sec_4_dis_one_ai1} -- \eqref{sec_4_dis_one_ai4} and the inequality in \eqref{sec_4_dis_one_ai5} -- \eqref{sec_4_dis_one_ai8} then carry over verbatim and we again arrive at \eqref{sec_4_dis_one_ai9} -- \eqref{sec_4_dis_one_ai11}, which shows the convexity of the function $[0,1] \ni u \mapsto \mathcal{U}(\alpha_{u})$.
\end{proof}

\bibliographystyle{abbrv}
\bibliography{joint_biblio}

\begin{thebibliography}{10}

\bibitem{AmGi13}
L.~Ambrosio and N.~Gigli.
\newblock A user's guide to optimal transport.
\newblock In {\em Modelling and optimisation of flows on networks}, volume 2062
  of {\em Lecture Notes in Math.}, pages 1--155. Springer, Heidelberg, 2013.

\bibitem{AGS08}
L.~Ambrosio, N.~Gigli, and G.~Savar{\'e}.
\newblock {\em Gradient flows in metric spaces and in the space of probability
  measures}.
\newblock Lectures in Mathematics ETH Z\"urich. Birkh\"auser Verlag, Basel,
  second edition, 2008.

\bibitem{BaBaBeEd19a}
J.~Backhoff-Veraguas, D.~Bartl, M.~Beiglb\"{o}ck, and M.~Eder.
\newblock Adapted {W}asserstein distances and stability in mathematical
  finance.
\newblock {\em Finance Stoch.}, 24(3):601--632, 2020.

\bibitem{BaBeHuKa20}
J.~Backhoff-Veraguas, M.~Beiglb\"{o}ck, M.~Huesmann, and S.~K\"{a}llblad.
\newblock Martingale {B}enamou-{B}renier: a probabilistic perspective.
\newblock {\em Ann. Probab.}, 48(5):2258--2289, 2020.

\bibitem{BBST23}
J.~Backhoff-Veraguas, M.~Beiglb\"ock, W.~Schachermayer, and B.~Tschiderer.
\newblock {T}he structure of martingale {B}enamou--{B}renier in
  $\mathbb{R}^{d}$.
\newblock {\em
  \href{https://arxiv.org/abs/2306.11019}{\textnormal{arXiv:12306.11019}}},
  2023.

\bibitem{Ba83}
R.~Bass.
\newblock Skorokhod embedding via stochastic integrals.
\newblock In J.~Az{\'e}ma and M.~Yor, editors, {\em S{\'e}minaire de
  {Probabilit{\'e}s} {XVII} 1981/82}, number 986 in Lecture {Notes} in
  {Mathematics}, pages 221--224. Springer, 1983.

\bibitem{BeCoHu14}
M.~{Beiglb{\"o}ck}, A.~Cox, and M.~Huesmann.
\newblock Optimal transport and {S}korokhod embedding.
\newblock {\em Invent. Math.}, 208(2):327--400, 2017.

\bibitem{BeHePe12}
M.~Beiglb{\"o}ck, P.~{Henry-Labord{\`e}re}, and F.~Penkner.
\newblock Model-independent bounds for option prices: A mass transport
  approach.
\newblock {\em Finance Stoch.}, 17(3):477--501, 2013.

\bibitem{BeHeTo15}
M.~Beiglb\"{o}ck, P.~Henry-Labord\`ere, and N.~Touzi.
\newblock Monotone martingale transport plans and {S}korokhod embedding.
\newblock {\em Stochastic Process. Appl.}, 127(9):3005--3013, 2017.

\bibitem{BeJoMaPa21b}
M.~Beiglb{\"o}ck, B.~Jourdain, W.~Margheriti, and G.~Pammer.
\newblock Monotonicity and stability of the weak martingale optimal transport
  problem.
\newblock {\em
  \href{https://arxiv.org/abs/2109.06322}{\textnormal{arXiv:2109.06322}}},
  2022.

\bibitem{BeJu16}
M.~Beiglb{\"o}ck and N.~Juillet.
\newblock On a problem of optimal transport under marginal martingale
  constraints.
\newblock {\em Ann. Probab.}, 44(1):42--106, 2016.

\bibitem{BeNuSt19}
M.~Beiglb\"{o}ck, M.~Nutz, and F.~Stebegg.
\newblock Fine properties of the optimal {S}korokhod embedding problem.
\newblock {\em J. Eur. Math. Soc. (JEMS)}, 24(4):1389--1429, 2022.

\bibitem{BeBr99}
J.~Benamou and Y.~Brenier.
\newblock A numerical method for the optimal time-continuous mass transport
  problem and related problems.
\newblock In {\em Monge {A}mp{\`e}re equation: applications to geometry and
  optimization ({D}eerfield {B}each, {FL}, 1997)}, volume 226 of {\em Contemp.
  Math.}, pages 1--11. Amer. Math. Soc., Providence, RI, 1999.

\bibitem{BoNu13}
B.~Bouchard and M.~Nutz.
\newblock Arbitrage and duality in nondominated discrete-time models.
\newblock {\em The Annals of Applied Probability}, 25(2):823--859, 2015.

\bibitem{Br87}
Y.~Brenier.
\newblock D\'ecomposition polaire et r\'earrangement monotone des champs de
  vecteurs.
\newblock {\em C. R. Acad. Sci. Paris S\'er. I Math.}, 305(19):805--808, 1987.

\bibitem{Br91}
Y.~Brenier.
\newblock Polar factorization and monotone rearrangement of vector-valued
  functions.
\newblock {\em Comm. Pure Appl. Math.}, 44(4):375--417, 1991.

\bibitem{CaLaMa14}
L.~{Campi}, I.~{Laachir}, and C.~{Martini}.
\newblock {Change of numeraire in the two-marginals martingale transport
  problem}.
\newblock {\em Finance Stoch.}, 21(2):471--486, June 2017.

\bibitem{ChKiPrSo20}
P.~Cheridito, M.~Kiiski, D.~J. Pr\"{o}mel, and H.~M. Soner.
\newblock Martingale optimal transport duality.
\newblock {\em Math. Ann.}, 379(3-4):1685--1712, 2021.

\bibitem{CoObTo19}
A.~M.~G. Cox, J.~Ob\l~\'{o}j, and N.~Touzi.
\newblock The {R}oot solution to the multi-marginal embedding problem: an
  optimal stopping and time-reversal approach.
\newblock {\em Probab. Theory Related Fields}, 173(1-2):211--259, 2019.

\bibitem{DeTo17}
H.~De~March and N.~Touzi.
\newblock Irreducible convex paving for decomposition of multidimensional
  martingale transport plans.
\newblock {\em Ann. Probab.}, 47(3):1726--1774, 2019.

\bibitem{DoSo12}
Y.~Dolinsky and H.~M. Soner.
\newblock Martingale optimal transport and robust hedging in continuous time.
\newblock {\em Probab. Theory Relat. Fields}, 160(1-2):391--427, 2014.

\bibitem{Fo22a}
H.~F\"{o}llmer.
\newblock Optimal couplings on {W}iener space and an extension of {T}alagrand's
  transport inequality.
\newblock In {\em Stochastic analysis, filtering, and stochastic optimization},
  pages 147--175. Springer, Cham, 2022.

\bibitem{GaHeTo13}
A.~Galichon, P.~{Henry-Labord{\`e}re}, and N.~Touzi.
\newblock A stochastic control approach to no-arbitrage bounds given marginals,
  with an application to lookback options.
\newblock {\em Ann. Appl. Probab.}, 24(1):312--336, 2014.

\bibitem{GhKiLi19}
N.~Ghoussoub, Y.-H. Kim, and T.~Lim.
\newblock Structure of optimal martingale transport plans in general
  dimensions.
\newblock {\em Ann. Probab.}, 47(1):109--164, 2019.

\bibitem{GhKiLi20}
N.~Ghoussoub, Y.-H. Kim, and T.~Lim.
\newblock Optimal {B}rownian stopping when the source and target are radially
  symmetric distributions.
\newblock {\em SIAM J. Control Optim.}, 58(5):2765--2789, 2020.

\bibitem{GhKiPa19}
N.~Ghoussoub, Y.-H. Kim, and A.~Z. Palmer.
\newblock P{DE} methods for optimal {S}korokhod embeddings.
\newblock {\em Calc. Var. Partial Differential Equations}, 58(3):Paper No. 113,
  31, 2019.

\bibitem{GuLo21}
I.~Guo and G.~Loeper.
\newblock Path dependent optimal transport and model calibration on exotic
  derivatives.
\newblock {\em The Annals of Applied Probability}, 31(3):1232--1263, 2021.

\bibitem{GuLoWa19}
I.~Guo, G.~Loeper, and S.~Wang.
\newblock Local volatility calibration by optimal transport.
\newblock In {\em 2017 MATRIX Annals}, pages 51--64. Springer, 2019.

\bibitem{HeObSpTo12}
P.~Henry-Labord\`ere, J.~Ob\l\'oj, P.~Spoida, and N.~Touzi.
\newblock The maximum maximum of a martingale with given {$n$} marginals.
\newblock {\em Ann. Appl. Probab.}, 26(1):1--44, 2016.

\bibitem{HoNe12}
D.~Hobson and A.~Neuberger.
\newblock Robust bounds for forward start options.
\newblock {\em Math. Finance}, 22(1):31--56, 2012.

\bibitem{HuTr17}
M.~Huesmann and D.~Trevisan.
\newblock A {B}enamou-{B}renier formulation of martingale optimal transport.
\newblock {\em Bernoulli}, 25(4A):2729--2757, 2019.

\bibitem{KaTaTo15}
S.~K\"{a}llblad, X.~Tan, and N.~Touzi.
\newblock Optimal {S}korokhod embedding given full marginals and
  {A}z\'{e}ma-{Y}or peacocks.
\newblock {\em Ann. Appl. Probab.}, 27(2):686--719, 2017.

\bibitem{Ka42}
L.~Kantorovich.
\newblock On the translocation of masses.
\newblock {\em C. R. (Doklady) Acad. Sci. URSS (N.S.)}, 37:199--201, 1942.

\bibitem{Lo18}
G.~Loeper.
\newblock Option pricing with linear market impact and nonlinear black--scholes
  equations.
\newblock {\em The Annals of Applied Probability}, 28(5):2664--2726, 2018.

\bibitem{Mc94}
R.~McCann.
\newblock A convexity theory for interacting gases and equilibrum crystals.
\newblock {\em PhD thesis, Princeton University}, 1994.

\bibitem{Mc95}
R.~McCann.
\newblock Existence and uniqueness of monotone measure-preserving maps.
\newblock {\em Duke Math. J.}, 80(2):309--323, 1995.

\bibitem{McCgas}
R.~J. McCann.
\newblock A convexity principle for interacting gases.
\newblock {\em Adv. Math.}, 128(1):153--179, 1997.

\bibitem{Mo81}
G.~Monge.
\newblock Memoire sur la theorie des deblais et des remblais.
\newblock {\em Histoire de l'acad\'emie {R}oyale des {S}ciences de {P}aris},
  1781.

\bibitem{ObSi17}
J.~Ob{\l}{\'o}j and P.~Siorpaes.
\newblock {Structure of martingale transports in finite dimensions}.
\newblock arXiv:1702.08433, 2017.

\bibitem{ObSp14}
J.~Ob{\l}{\'o}j, P.~Spoida, and N.~Touzi.
\newblock Martingale inequalities for the maximum via pathwise arguments.
\newblock In {\em In {M}emoriam {M}arc {Y}or-S{\'e}minaire de
  {P}robabilit{\'e}s XLVII}, pages 227--247. Springer, 2015.

\bibitem{Sa15}
F.~Santambrogio.
\newblock {\em Optimal transport for applied mathematicians}, volume~87 of {\em
  Progress in Nonlinear Differential Equations and their Applications}.
\newblock Birkh\"auser/Springer, Cham, 2015.
\newblock Calculus of variations, PDEs, and modeling.

\bibitem{TaTo13}
X.~Tan and N.~Touzi.
\newblock Optimal transportation under controlled stochastic dynamics.
\newblock {\em Ann. Probab.}, 41(5):3201--3240, 2013.

\bibitem{Vi03}
C.~Villani.
\newblock {\em Topics in optimal transportation}, volume~58 of {\em Graduate
  Studies in Mathematics}.
\newblock American Mathematical Society, Providence, RI, 2003.

\bibitem{Vi09}
C.~Villani.
\newblock {\em Optimal Transport. Old and New}, volume 338 of {\em Grundlehren
  der mathematischen Wissenschaften}.
\newblock Springer, 2009.

\end{thebibliography}
\end{document}